\documentclass[11pt]{amsart}
\usepackage[T1]{fontenc}
\usepackage[english]{babel}
\usepackage{amscd,amsmath,amsthm,amssymb,graphics}
\usepackage{lmodern,pst-node}
\usepackage{pstcol,pst-plot,pst-3d}
\usepackage{multicol}
\usepackage{epic,eepic}
\usepackage{amsfonts,amssymb,amscd,amsmath,enumitem,verbatim}
\psset{unit=0.7cm,linewidth=0.8pt,arrowsize=2.5pt 4}

\newpsstyle{fatline}{linewidth=1.5pt}
\newpsstyle{fyp}{fillstyle=solid,fillcolor=verylight}
\definecolor{verylight}{gray}{0.97}
\definecolor{light}{gray}{0.9}
\definecolor{medium}{gray}{0.85}
\definecolor{dark}{gray}{0.6}

\def\frk{\frak}               

\def\Phi{{\frk n}}
\def\Phi{{\frk N}}
\def\opn#1#2{\def#1{\operatorname{#2}}} 
\opn\chara{char} \opn\length{\ell} \opn\pd{pd} \opn\rk{rk}
\opn\projdim{proj\,dim} \opn\injdim{inj\,dim} \opn\rank{rank}
\opn\depth{depth} \opn\grade{grade} \opn\height{height}
\opn\embdim{emb\,dim} \opn\codim{codim}

\opn\Tr{Tr} \opn\bigrank{big\,rank}
\opn\superheight{superheight}\opn\lcm{lcm}
\opn\trdeg{tr\,deg}
\opn\reg{reg} \opn\lreg{lreg} \opn\ini{in} \opn\lpd{lpd}
\opn\size{size}\opn\bigsize{bigsize}
\opn\cosize{cosize}\opn\bigcosize{bigcosize}
\opn\sdepth{sdepth}\opn\sreg{sreg}
\opn\link{link}\opn\fdepth{fdepth}
\opn\deg{deg}
\opn\max{max}
\opn\indeg{indeg}
\opn\min{min}
\opn\psln{psln}
\opn\div{div} \opn\Div{Div} \opn\cl{cl} \opn\Cl{Cl}
\let\epsilon\varepsilon
\let\phi=\varphi
\let\kappa=\varkappa

\opn\Spec{Spec} \opn\Supp{Supp} \opn\supp{supp} \opn\Sing{Sing}
\opn\Ass{Ass} \opn\Min{Min}\opn\Mon{Mon} \opn\dstab{dstab} \opn\astab{astab}
\opn\Syz{Syz}
\opn\Ann{Ann} \opn\Rad{Rad} \opn\Soc{Soc}
\opn\Im{Im} \opn\Ker{Ker} \opn\Coker{Coker} \opn\Am{Am}
\opn\Hom{Hom} \opn\Tor{Tor} \opn\Ext{Ext} \opn\End{End}
\opn\Aut{Aut} \opn\id{id}

\opn\nat{nat}
\opn\pff{pf}
\opn\Pf{Pf} \opn\GL{GL} \opn\SL{SL} \opn\mod{mod} \opn\ord{ord}
\opn\Gin{Gin} \opn\Hilb{Hilb}\opn\sort{sort}
\opn\initial{init}
\opn\ende{end}
\opn\height{height}
\opn\depth{depth}
\opn\type{type}
\opn\ldim{ldim}
\opn\lk{lk}
\opn\del{del}

\opn\aff{aff} \opn\con{conv} \opn\relint{relint} \opn\st{st}
\opn\lk{lk} \opn\cn{cn} \opn\core{core} \opn\vol{vol}
\opn\link{link} \opn\star{star}\opn\lex{lex}
\opn\gr{gr}

\def\pot#1#2{#1[\kern-0.28ex[#2]\kern-0.28ex]}

\opn\dirlim{\underrightarrow{\lim}}
\opn\inivlim{\underleftarrow{\lim}}
\def\Implies{\ifmmode\Longrightarrow \else
        \unskip${}\Longrightarrow{}$\ignorespaces\fi}
\def\implies{\ifmmode\Rightarrow \else
        \unskip${}\Rightarrow{}$\ignorespaces\fi}
\def\iff{\ifmmode\Longleftrightarrow \else
        \unskip${}\Longleftrightarrow{}$\ignorespaces\fi}

\let\:=\colon
 \theoremstyle{plain}
\newtheorem{Theorem}{Theorem}[section]
 \newtheorem{Lemma}[Theorem]{Lemma}
 
 \newtheorem{Proposition}[Theorem]{Proposition}

 \theoremstyle{definition}
 \newtheorem{Definition}[Theorem]{Definition}

 \newtheorem{Example}[Theorem]{Example}

\DeclareMathOperator{\g}{\mathcal{G}}

\let\epsilon\varepsilon
\let\kappa=\varkappa
\def\qed{\ifhmode\textqed\fi
      \ifmmode\ifinner\quad\qedsymbol\else\dispqed\fi\fi}
\def\textqed{\unskip\nobreak\penalty50
       \hskip2em\hbox{}\nobreak\hfil\qedsymbol
       \parfillskip=0pt \finalhyphendemerits=0}
\def\dispqed{\rlap{\qquad\qedsymbol}}

\opn\dis{dis}
\def\pnt{{\raise0.5mm\hbox{\large\bf.}}}

\opn\Lex{Lex}

\begin{document}

\author[Mafi, Naderi and Saremi]{ Amir Mafi, Dler Naderi and Hero Saremi}
\title{ Vertex decomposability and weakly polymatroidal ideals}

\address{Amir Mafi, Department of Mathematics, University Of Kurdistan, P.O. Box: 416, Sanandaj, Iran.}
\email{A\_Mafi@ipm.ir}
\address{Dler Naderi, Department of Mathematics, University of Kurdistan, P.O. Box: 416, Sanandaj,
Iran.}
\email{dler.naderi65@gmail.com}
\address{Hero Saremi, Department of Mathematics, Sanandaj Branch, Islamic Azad University, Sanandaj, Iran.}

\email{h-saremi@iausdj.ac.ir and hero.saremi@gmail.com}

\begin{abstract}
Let $K$ be a field and $R=K[x_1,\ldots, x_n]$ be the polynomial ring in $n$ variables over a field $K$. Let $\Delta$ be a simplicial complex on $n$ vertices and $I=I_{\Delta}$ be its Stanley-Reisner ideal.
In this paper, we show that if $I$ is a matroidal ideal then the following conditions are equivalent: $(i)$ $\Delta$ is sequentially Cohen-Macaulay; $(ii)$ $\Delta$ is shellable; $(iii)$ $\Delta$ is vertex decomposable. Also, if $I$ is a minimally generated by $u_1,\ldots,u_s$ such that $s\leq 3$ or $\supp(u_i)\cup\supp(u_j)=\{x_1,\ldots,x_n\}$ for all $i\neq j$, then $\Delta$ is vertex decomposable. Furthermore, we prove that
if $I$ is a monomial ideal of degree $2$ then $I$ is weakly polymatroidal if and only if $I$ has linear quotients if and only if $I$ is vertex splittable.
\end{abstract}


\maketitle
\section*{Introduction}
Throughout this paper, we assume that $R=K[x_1,\ldots,x_n]$ is the polynomial ring in $n$ variables over the field $K$ and $I$ is a monomial ideal of $R$. We denote, as usual, by $\g(I)$ the unique minimal set of monomial generators of $I$. For each monomial $u= x_1^{a_1} \cdots x_n^{a_n}$, we define the support of $u$ to be $\supp(u) = \{ x_i | a_i > 0 \}$.

The ideal $I$ has linear quotients if there is an ordering
$u_1,\ldots, u_s$ of monomials belonging to $\g(I)$ such that for each integer $1 < j \leq s$, the colon
ideal $(u_1,\ldots, u_{j-1}) : u_j$ is generated by a subset of $\{x_1, \ldots, x_n \}$.
Ideals with linear quotients were introduced by Herzog and Takayama in \cite{HT}.
Conca and Herzog \cite{CH} proved that if a monomial ideal $I$ generated in one degree has linear quotients, then $I$ has a linear resolution.
 A monomial ideal I is called a polymatroidal ideal, if it is generated in a single degree with the exchange property that for each two elements $u, v \in \g(I)$ such that $\deg_{x_i}(u) > \deg_{x_i}(v)$ for some $i$, there exists an integer $j$ such that $\deg_{x_j} (u) < \deg_{x_j} (v)$ and $x_{j}(u/x_{i}) \in \g(I)$. The polymatroidal ideal $I$ is called matroidal if $I$ is generated by square-free monomials.
 Herzog and Takayama \cite{HT} proved that polymatroidal ideals have linear quotients and so they have linear resolution.
Herzog, Hibi and Zheng \cite{HHZ} proved that if $I$ is a monomial ideal generated in degree $2$, then $I$ has a linear resolution if and only if $I$ has linear quotients if and only if each power of $I$ has a linear resolution.

Kokubo and Hibi \cite{KH} introduced weakly polymatroidal ideals generated in the same
degree as a generalization of polymatroidal ideals.
Mohammadi and Moradi \cite{MM} extended the definition of weakly polymatroidal to ideals which are not necessarily generated
in one degree, and they are defined as follows:
A monomial ideal $I$ of $R$ is called weakly polymatroidal if for every two monomials
$v = x_{1}^{b_1} \cdots x_{n}^{b_n}<_{\lex}u = x_{1}^{a_1} \cdots x_{n}^{a_n}$ belonging to $\g(I)$ such that $a_1=b_1, \dots, a_{t-1}=b_{t-1}$ and $a_t > b_t$ for some $t$, there exists $j>t$ such that $x_{t} (v/x_{j}) \in I$. Kokubo and Hibi demonstrated that every weakly polymatroidal ideal generated in a single degree possesses linear quotients \cite{KH}. This result is further supported by Theorem 1.8 in \cite{MM}. Note that, in this definition, we consider the lexicographic monomial order $<_{\lex}$ on $R$ induced by the ordering $x_1>x_2>\ldots>x_n$ of the variables.

For a square-free monomial ideal $I$ of $R$, we may consider the simplicial complex $\Delta$ for which $I=I_{\Delta}$ is the Stanley-Reisner ideal of $\Delta$ and $K[\Delta]=R/I_{\Delta}$ is the Stanley-Reisner ring.
Eagon and Reiner \cite{ER} proved that $R/I$ is Cohen-Macaulay if and only if the Alexander dual $I^{\vee}$ has a linear resolution.
Herzog and Hibi \cite{HH1} generalized the notion of linear resolution to componentwise linearity and they proved that  $R/I$ is sequentially Cohen-Macaulay if and only if the Alexander dual $I^{\vee}$ is componentwise linear. It is known that if  $I$ has linear quotients, then $I$ is componentwise linear. Herzog, Hibi and Zheng \cite{HHZ1} proved that the simplicial complex $\Delta$ is shellable if and only if $I^{\vee}$ has linear quotients. Recently Moradi and Khosh-Ahang \cite{MK} proved that the simplicial complex $\Delta$ is vertex decomposable if and only if $I^{\vee}$ is vertex splittable and also one concludes that every vertex splittable ideal has linear quotients. So we have the following implications:
$$vertex~ decomposable \Longrightarrow shellable \Longrightarrow sequentially~Cohen-Macaulay.$$

Both implications are known to be strict.
The equivalence between the sequentially Cohen-Macaulay property, the shellable property and the vertex decomposable property have been studied for some families of graphs: bipartite graphs \cite{VV, V}, chordal graphs \cite{FV1, W}, and very well-covered graphs \cite{MMC} and Cactus graphs \cite{MKY}.

In this paper, we show that the above implications in the following cases are equivalent: $(i)$ if $I$ is a matroidal ideal $(ii)$ if $I$ is a square-free monomial ideal minimally generated by $u_1,\dots,u_s$ such that $s\leq 3$ or $\supp(u_i)\cup\supp(u_j)=\{x_1,\ldots,x_n\}$ for all $i\neq j$.
In the end we prove that if $I$ is a monomial ideal of $R$ with $\deg(I)\leq 2$, then $I$ is weakly polymatroidal if and only if $I$ has linear quotients if and only if $I$ is a vertex splittable, where $\deg(I)=\max\{\deg(u)\mid u\in\g(I)\}$. 

For any unexplained notion or terminology, we refer the reader to \cite{HH}. Several explicit examples were  performed with help of the computer algebra system Macaulay2 \cite{G}.

\section{Preliminaries}

In this section, we recall some definitions and properties that will be used in this article.
 Let $\Delta$ be a simplicial complex on the vertex set $V = \{x_1,\ldots, x_n \}$. Members of $\Delta$ are called faces of $\Delta$ and a facet of $\Delta$ is a maximal face of $\Delta$ with respect to inclusion. The dimension of a face $F$ is $\vert F \vert -1$ and the dimension of a complex $\Delta$ is the maximum of the dimensions of its facets. If all facets of $\Delta$ have the same dimension, then $\Delta$ is called {\it pure} and also $\Delta$ is called a {\it simplex} when it has a unique facet. If $\Delta$ is a simplicial complex with facets $F_1,\ldots, F_t$, we denote $\Delta$ by $ \langle F_1, \ldots, F_t \rangle $, and $\{F_1, \ldots, F_t \}$ is called the facet set of $\Delta$.

 For a given simplicial complex $\Delta$ on $V$, we define $\Delta^{\vee}$ by $\Delta^{\vee} = \{V \setminus A~ | ~A \notin \Delta \}$.
The simplicial complex $\Delta^{\vee}$ is called the {\it Alexander dual} of $\Delta$.
For every subset $F \subseteq V$, we set $x_{ F} = \prod_{x_j \in F} x_{j}$. The {\it Stanley-Reisner} ideal of $\Delta$ over $K$ is the ideal $I_{\Delta}$ of $R$ which is generated by the square-free monomials $x_F$ with $F \notin \Delta$. Let $I$ be an arbitrary square-free monomial ideal. Then there is a unique simplicial complex $\Delta$ such that $I = I_{\Delta}$. For simplicity,
we often write $I^{\vee}$ to denote the ideal $I_{\Delta^{\vee}}$, and we call it the {\it Alexander dual} of $I$.
If $I$ is a square-free monomial ideal $I = \cap_{i=1}^{t} \frak{p_{i}}$, where each of the $\frak{p_{i}}$ is a monomial prime ideal of $I$, then the ideal $I^{\vee}$ is minimally generated by monomials $u_i = \prod_{x_{j} \in \frak{p_{i}}} x_{j}$.

For the simplicial complex $\Delta$ and the face $F \in \Delta$, we can create two new simplicial complexes. The {\it deletion} of $F$ from $\Delta$ is $\del_{\Delta}(F) = \{ A \in \Delta \vert F\cap A=\emptyset\} $. The {\it link} of $F$ in $\Delta$ is $\lk_{\Delta}(F) =\{ A\in  \Delta \vert F \cap A =\emptyset, A\cup F\in\Delta \}$. If $F =\{x \}$, we write $\del_{\Delta} x$ (resp. $\lk_{\Delta}x$) instead of $\del_{\Delta}(\{x\})$ (resp. $\lk_{\Delta}(\{x\})$); see \cite{HH} for more detail informations.

Vertex decompositions were introduced in the pure case by Provan and Billera \cite{PB}, and extended to non-pure complexes by Bj\"orner and Wachs \cite{BW}.
A simplicial complex $\Delta$ is recursively defined to be {\it vertex decomposable} if it is either simplex or else has some vertex $x$ such that
\begin{enumerate}
\item
both $\del_{\Delta} x$ and $\lk_{\Delta}x$ are vertex decomposable, and
\item
there is no face of $\lk_{\Delta}x$ which is also a facet of $\del_{\Delta} x$.
\end{enumerate}
A vertex $x$ which satisfies in condition (2) is called a {\it shedding vertex}.
An ideal $I$ is called vertex decomposable if $I=I_{\Delta}$ which is $\Delta$ is vertex decomposable.

Moradi and Khosh-Ahang \cite{MK} defined the notion of vertex splittable ideal which is an algebraic analog of the vertex decomposability property of a simplicial complex and was defined as follows:
\begin{Definition}\label{D1}
A monomial ideal $I$ of $R$ is called vertex splittable if it can be obtained by the following recursive procedure:
\begin{enumerate}
\item
if $v$ is a monomial and $I=(v)$, $I = (0)$ or $I = R$, then I is a vertex splittable ideal;
\item
 if there is a variable $x\in V$ and vertex splittable ideals $I_1$ and $I_2$ of $K[V\setminus\{x\}]$ so that $I = xI_1+I_2$, $I_2 \subseteq I_1$ and $\g(I)$ is the disjoint union of $\g(xI_1)$ and $\g(I_2)$, then $I$ is a vertex splittable ideal.
\end{enumerate}
\end{Definition}
By the above notations if $I = xI_1 + I_2$ is a vertex splittable ideal, then $xI_1 + I_2$ is called a vertex splitting for $I$ and $x$ is called a splitting vertex for $I$.

A simplicial complex $\Delta$ is {\it shellable} if the facets of $\Delta$ can be ordered, say $F_1,\ldots,F_s$, such that for all $1\leq i<j\leq s$, there exists some $x\in F_j\setminus F_i$ and some $l\in\{1,2,\ldots,j-1\}$ with $F_j\setminus F_l=\{x\}$. An ideal $I$ is called shellable if $I=I_{\Delta}$ which is $\Delta$ is shellable.
Let $I$ be a monomial ideal of $R$ with $\g(I)=\{u_1,\ldots,u_r\}$. We say that $I$ has linear quotients if there is an ordering $u_1,u_2,\ldots,u_r$ such that for each $2\leq i\leq r$ the colon ideal $(u_1,\ldots,u_{i-1}):u_i$ is generated by a subset $\{x_1,\ldots,x_n\}$ (see \cite{HT}).
Note that if $I=(u_1,\ldots,u_r)$ is a monomial ideal with linear quotients, then the Castelnuovo-Mumford regularity $\reg(I)=\max\{\deg(u_i)| i=1,2,\ldots,r\}$ (see \cite[Lemma 4.1]{CH}).
Herzog, Hibi and Zheng \cite[Theorem 1.4]{HHZ1} proved that if $\Delta$ is a simplicial complex for which $I=I_{\Delta}$, then $I$ is shellable if and only if $I^{\vee}$ has linear quotients.

Stanley \cite{S} defined a graded $R$-module $M$ to be sequentially Cohen-Macaulay if there exists a finite filtration of graded $R$-modules
$0=M_0\subset M_1\subset\dots\subset M_r=M$ such that each $M_{i}/M_{i-1}$ is Cohen-Macaulay, and the Krull dimensions of the quotients are increasing:
$\dim(M_1/M_0)<\dim(M_2/M_1)<\dots<\dim(M_r/M_{r-1})$. In particular, we call the monomial ideal $I$ to be sequentially Cohen-Macaulay if $R/I$ is
sequentially Cohen-Macaulay.

For a homogeneous ideal $I$, we write $(I_i)$ to denote the ideal generated by the degree $i$ elements of $I$.  A monomial ideal $I$ is componentwise linear if $(I_i)$ has a linear resolution for all $i$ (see \cite{HH1}).

 Herzog and Hibi \cite{HH1} proved that the square-free monomial ideal $I$ is sequentially Cohen-Macaulay if and only if the Alexander dual $I^{\vee}$ is componentwise linear. Francisco and Van Tuyl \cite[Proposition 2.6]{FV} proved that if $I$ is a homogeneous ideal with linear quotients, then $I$ is a componentwise linear. Hence for a monomial ideal $I$ we have the following implications:
\[ vertex~ splittable \Longrightarrow linear~ quotients \Longrightarrow componentwise~ linear. \]
Both implications are known to be strict.

\section{ Main Results}
We start this section by the following lemma.

\begin{Lemma}\label{1}
Let $I$ be a matroidal ideal in $R$. Then $I$ is vertex splittable.
\end{Lemma}
\begin{proof}
We proceed by induction on $n$. If $n=1$, then $I$ is a principal ideal. Therefore, in this case, the assertion is trivial. Suppose that $n>1$.
By the proof of \cite[Theorem 1.1]{BH} we may assume $I=x_{1}I_{1}+I_{2}$ such that $I_{1}$ and $I_{2}$ are matroidal ideals in $K[V \setminus \{x_1\}]$ and $I_2 \subseteq I_1$, where $V =\{x_1,\ldots, x_n \}$.

By induction hypothesis the ideals $I_1$ and $I_2$ of $K[V \setminus \{x_1\}]$ are vertex splittable. Hence $I$ is vertex splittable ideal, as required.
\end{proof}

By the above lemma and \cite[Lemma 1.3]{HT} we get that every matroidal ideal have vertex splittable and so it has linear quotients and linear resolution.

The following example shows that the converse of the above lemma does not hold.
\begin{Example}
Let $I=(x_1x_2,x_2x_3,x_3x_4)$ be a monomial ideal of $R=K[x_1,x_2,x_3,x_4]$. Since $I=x_2(x_1,x_3)+(x_3x_4)$, it follows that $I$ is vertex splittable ideal. To see why $I$ is not matroidal ideal, note that  $\deg_{x_2}(x_1x_2)>\deg_{x_2}(x_3x_4)$. However, neither $x_3 (\frac{x_1x_2}{x_2})$ nor $x_4 (\frac{x_1x_2}{x_2})$ belongs to $I$. Hence $I$ is not a matroidal ideal.
\end{Example}

Hamaali, the first author and the third author \cite[Corollary 2.9]{HMS} proved that if $\Delta$ is a simplicial complex and $I$ is a matroidal ideal of $R$ such that $I=I_{\Delta}$, then $\Delta$ is sequentially Cohen-Macaulay if and only if $\Delta$ is shellable. In the following we show that these conditions are equivalent to vertex decomposability.
\begin{Theorem}\label{T0}
Let $\Delta$ be a simplicial complex and $I$ be a matroidal ideal of $R$ such that $I=I_{\Delta}$. Then the following conditions are equivalent:
\begin{enumerate}
\item[(i)] $\Delta$ is sequentially Cohen-Macaulay;
\item[(ii)] $\Delta$ is shellable;
\item[(iii)] $\Delta$ is vertex decomposable.
\end{enumerate}
\end{Theorem}

\begin{proof}
The implications $(iii)\Longrightarrow(ii)\Longrightarrow(i)$ always hold. Hence it suffices to prove $(i)\Longrightarrow(iii)$. Assume that
$\Delta$ is sequentially Cohen-Macaulay. We may assume that $I=x_1I_1+I_2$, where $I_2\subseteq I_1$ are matroidal ideals. By using the proof of \cite[Theorem 2.8]{HMS}, $I_1$ and $I_2$ are sequentially Cohen-Macaulay and so $I_1^{\vee}$ and $I_2^{\vee}$ have componentwise linear resolution. Since $I=(x_1,I_2)\cap I_1$, it follows that $I^{\vee}=x_1I_2^{\vee}+I_1^{\vee}$ and $I_1^{\vee}\subseteq I_2^{\vee}$. Now by induction hypothesis on $n$, $I_1^{\vee}$ and $I_2^{\vee}$ are vertex splittable. Hence $I^{\vee}$ is vertex splittable and so $\Delta$ is vertex decomposable, as required.
\end{proof}

\begin{Theorem}\cite[Theorem 2.5]{M}\label{T00}
If $\Delta$ is a simplicial complex such that $I=I_{\Delta}$ is weakly polymatroidal with respect to $x_1>x_2>\dots>x_n$, then $\Delta^{\vee}$ is vertex decomposable.
\end{Theorem}
From Theorem \ref{T00} and Definition \ref{D1} we can conclude that if the square-free monomial ideal $I$ is weakly polymatroidal, then $I$ is vertex splittable.

\begin{Lemma}\label{L1}
Let $\Delta$ be a simplicial complex and let $I=I_{\Delta}$ be minimally generated by square-free monomials $u_1,u_2$. Then $\Delta$ is vertex decomposable.
\end{Lemma}

\begin{proof}
It is clear that $I^{\vee}=\frak{p}_1\cap\frak{p}_2$ such that $\g(\frak{p_i})=\supp(u_i)$ for $i=1,2$. Therefore by \cite[Theorem 2.3]{MM} $I^{\vee}$ is square-free weakly polymatroidal and so $I^{\vee}$ is vertex splittable. Hence $\Delta$ is vertex decomposable, as required.
\end{proof}

Francisco and Van Tuyl \cite[Corollary 6.6]{FV} showed that if $\Delta$ is a simplicial complex with $I=I_{\Delta}$ is a minimally generated by square-free monomials $u_1,\ldots,u_s$, then $\Delta$ is  sequentially Cohen-Macaulay in the following cases $(i)$ if $s\leq 3$; $(ii)$ if $\supp(u_i)\cup\supp(u_j)=\{x_1,\ldots,x_n\}$ for all $i\neq j$. In the following we show that the above assertion is true for vertex decomposability.

\begin{Theorem}\label{T1}
Let $\Delta$ be a simplicial complex and let $I=I_{\Delta}$ be a minimally generated by square-free monomials $u_1,\ldots,u_s$. Then $\Delta$ is vertex decomposable in the following cases:
\begin{enumerate}
\item[(i)] if $s\leq 3$;
\item[(ii)]if $\supp(u_i)\cup\supp(u_j)=\{x_1,\ldots,x_n\}$ for all $i\neq j$.
\end{enumerate}

\end{Theorem}

\begin{proof} Consider the case $(i)$. It is enough to show that $I^{\vee}$ is a vertex splittable ideal. If $s=1$, then the result is clear. If $s=2$, then by Lemma \ref{L1} we get the result. Suppose that $s=3$ and $I=(u_1,u_2,u_3)$. Then $I^{\vee}=\frak{p}_1\cap\frak{p}_2\cap\frak{p}_3$ such that $\g(\frak{p_i})=\supp(u_i)$ for $i=1,2,3$. We consider the following cases:\\
{\bf Case 1}: if $\g(\frak{p}_{i})\cap\g(\frak{p}_{j})=\emptyset$ for $1\leq i\ne j\leq 3$, then $I^{\vee}$ is a matroidal ideal. Therefore by Lemma \ref{1},
$I^{\vee}$ is vertex splittable.\\
{\bf Case 2}: if $\g(\frak{p}_{i_1})\cap G(\frak{p}_{i_j})=\emptyset$ for $j=2,3$ and $\{i_1,i_2,i_3\}=\{1,2,3\}$, then by \cite[Theorem 2.4]{MM} $I^{\vee}$ is weakly polymatroidal and so it is vertex splittable.\\
{\bf Case 3}: if $\g(\frak{p}_{i_1})\cap \g(\frak{p}_{i_2})\ne\emptyset$ and $\g(\frak{p}_{i_1})\cap\g(\frak{p}_{i_3})=\emptyset$, then $\g(\frak{p}_{i_2})\cap \g(\frak{p}_{i_3})\ne\emptyset$ on the otherwise we come back to the Case 2. Therefore we may consider  $I^{\vee}=x\frak{p}_{i_3}+\frak{q}_{i_1}\cap\frak{q}_{i_2}\cap\frak{p}_{i_3}$
 for some $x\in\g(\frak{p}_{i_1})\cap\g(\frak{p}_{i_2})$, $\{i_1,i_2,i_3\}=\{1,2,3\}$ and $\frak{p}_{i_j}=(x,\frak{q}_{i_j})$ for $j=1,2$. Hence $I^{\vee}$ is vertex splittable if and only if
$\frak{q}_{i_1}\cap\frak{q}_{i_2}\cap\frak{p}_{i_3}$ is vertex splittable. By continuing this fashion on $\frak{q}_{i_1}\cap\frak{q}_{i_2}\cap\frak{p}_{i_3}$ and using the Cases 1, 2 we obtain that $I^{\vee}$ is vertex splittable.\\
{\bf Case 4}: if $\g(\frak{p}_{1})\cap\g(\frak{p}_{2})\cap\g(\frak{p}_{3})=\emptyset$, then either $\g(\frak{p}_{i})\cap \g(\frak{p}_{j})\ne\emptyset$ for all $1\leq i\ne j\leq 3$ or we have one of the Cases 1,2,3. Thus it enough to consider $\g(\frak{p}_{i})\cap \g(\frak{p}_{j})\ne\emptyset$ for all $1\leq i\ne j\leq 3$. Suppose $x\in\g(\frak{p}_{1})\cap\g(\frak{p}_{2})$. Hence $I^{\vee}=x\frak{p}_{3}+\frak{q}_{1}\cap\frak{q}_{2}\cap\frak{p}_{3}$ and $\frak{p}_i=(x,\frak{q}_i)$ for $i=1,2$. Thus $I^{\vee}$ is vertex splittable if and only if $\frak{q}_{1}\cap\frak{q}_{2}\cap\frak{p}_{3}$ is vertex splittable. Therefore by continuing this argument we get to one of the above cases and so the result follows.\\
{\bf Case 5}: if $\g(\frak{p}_{1})\cap\g(\frak{p}_{2})\cap\g(\frak{p}_{3})\ne\emptyset$, then
 $I^{\vee}=xR+\frak{q}_{1}\cap\frak{q}_{2}\cap\frak{q}_{3}$ for some $x\in\g(\frak{p}_{1})\cap\g(\frak{p}_{2})\cap\g(\frak{p}_{3})$ and $\frak{p}_i=(x,\frak{q}_i)$ for $i=1,2,3$. Hence $I^{\vee}$ is vertex splittable if and only if
 $\frak{q}_{1}\cap\frak{q}_{2}\cap\frak{q}_{3}$ is vertex splittable. By continuing this fashion on $\frak{q}_{1}\cap\frak{q}_{2}\cap\frak{q}_{3}$ and using the above cases we get that $I^{\vee}$ is vertex splittable. This completes the proof of case $(i)$.

 Now, we consider the case $(ii)$ and we prove that $I^{\vee}$ is a vertex splittable ideal.
 we can assume $I^{\vee}=\frak{p}_1\cap\dots\cap\frak{p}_s$. If $x\in\cap_{i=1}^s\g(\frak{p}_i)$, then we have $I^{\vee}=xR+\frak{q}_1\cap\dots\cap\frak{q}_s$ where $\frak{q}_i$'s are monomial prime ideals such that $\frak{p}_i=(x,\frak{q}_i)$ for all $i=1,\ldots,s$. Therefore $I^{\vee}$ is vertex splittable if and only if $\frak{q}_1\cap\dots\cap\frak{q}_s$ is vertex splittable. By continuing this arguments we may assume that $\cap_{i=1}^s\g(\frak{p}_i)=\emptyset$.
 If $\g(\frak{p}_i)\cap \g(\frak{p}_j)=\emptyset$ for all $1\leq i\ne j\leq s$, then $I^{\vee}$ is a matroidal ideal and so by Lemma \ref{1}, we get $I^{\vee}$ vertex splittable. Now, let $x\notin\g(\frak{p}_j)$ for some $j$. By assumption $\g(\frak{p}_i)\cup \g(\frak{p}_j)=\{x_1,\ldots,x_n\}$ for all $i\neq j$, it follows that $x\in \g(\frak{p}_i)$ for all $i\ne j$. Therefore we get $I^{\vee}=x\frak{p}_j+\frak{q}_1\cap\dots\cap\frak{q}_{j-1}\cap\frak{p}_j\cap\frak{q}_{j+1}\cap\dots\cap\frak{q}_s$, where $\frak{q}_i$'s are monomial prime ideals with $\frak{p}_i=(x,\frak{q}_i)$ for all $i\ne j$, $\g(\frak{q}_r)\cup \g(\frak{q}_t)=\{x_1,\ldots,x_n\}\setminus\{x\}$ for all $r\ne t\in\{1,\ldots,s\}\setminus\{j\}$ and $\g(\frak{q}_i)\cup \g(\frak{p}_j)=\{x_1,\ldots,x_n\}\setminus\{x\}$ for all $i\ne j$. Hence $I^{\vee}$ is vertex splittable if and only if $\frak{q}_1\cap\dots\cap\frak{q}_{j-1}\cap\frak{p}_j\cap\frak{q}_{j+1}\cap\dots\cap\frak{q}_s$ is vertex splittable. Since 
 $\g(\frak{q}_r)\cup \g(\frak{q}_t)=\{x_1,\ldots,x_n\}\setminus\{x\}$ for all $r\ne t\in\{1,\ldots,s\}\setminus\{j\}$ and 
 $\g(\frak{q}_i)\cup \g(\frak{p}_j)=\{x_1,\ldots,x_n\}\setminus\{x\}$ for all $i\ne j$, by induction hypothesis on $n$, we get the result, as required.
\end{proof}

The following example shows that Theorem \ref{T1}, in general, does not hold for a monomial ideal minimally generated by $4$ monomial elements.
\begin{Example}
Let $I=(x_1x_2,x_2x_3,x_3x_4,x_1x_4)$ be a monomial ideal of $R=k[x_1,x_2,x_3,x_4]$. Then $I^{\vee}=(x_1x_3,x_2x_4)$ and so $\reg(I^{\vee})=3$. Thus $I$ is not sequentially Cohen-Macaulay and so it is not vertex decomposable.
\end{Example}

A graph $G$ is called chordal if each cycle of length$>3$ has a chord.
\begin{Proposition}\label{L2}
Let $I$ be a square-free monomial ideal of $R$ with $\deg(I)\leq 2$. Then the following conditions are equivalent:
\begin{enumerate}
\item[(i)] $I$ is (up to a relabeling of the variables of $R$ if necessarily) weakly polymatroidal;
\item[(ii)] $I$ is vertex splittable;
\item[(iii)] $I$ has linear quotients.
\end{enumerate}
\end{Proposition}

\begin{proof}
$(i)\Longrightarrow(ii)$. Let $I=(x_{i_1},\ldots,x_{i_r},u_1,\ldots,u_t)$, where $x_{i_1},\ldots,x_{i_r}$ are variables and $u_1,\ldots,u_t$ are square-free monomial elements of degree $2$. Since $I$ is weakly polymatroidal, by \cite[Theorem 1.3]{MM} we deduce that $I$ has a linear quotients. Also, it is clear that $I$ has a linear quotients if and only if  $J=(u_1,\ldots,u_t)$ has a linear quotients and as well as $I$ is vertex splittable if and only if  $J=(u_1,\ldots,u_t)$ is vertex splittable.
Since $J$ has a linear quotients, by applying \cite[Corollary 3.6]{RY} we conclude that  $J$ is weakly polymatroidal. Now, by using Theorem \ref{T00} we deduce that $J$ is vertex splittable and so is $I$.\\
$(ii)\Longrightarrow(iii)$ is known as before.\\
Now, we prove $(iii)\Longrightarrow(i)$. After relabeling the variables of $R$, we may assume $I=(x_1,\ldots,x_t,J)$ such that
$J$ is a square-free monomial ideal of single degree $2$ and $\supp(J)\cap\{x_1,\ldots,x_t\}=\emptyset$. Since $I$ has linear quotients, it immediately follows that $J$ has linear quotients.
Hence by \cite[Theorem 3.2]{HHZ} $J$ has a linear resolution. If, we consider $J$ as the edge ideal of a graph $G$, then by \cite[Theorem 6]{F}
$\overline{G}$ is chordal and so by \cite[Theorem 2.2]{M} $J$ is weakly polymatroidal. By definition, $I$ is weakly polymatroidal, as required.
\end{proof}

\begin{Lemma}\cite[Lemma 2.4]{PFY}\label{L3}
Let $I$ be a weakly polymatroidal ideal of $R$ which is generated in a single degree. Then $(I:x_1)$ satisfies in the same property.
\end{Lemma}
The above lemma was ordering by $x_1>x_2>\dots>x_n$. In the following example we show that if we use $x_2$ instead of $x_1$, then the result of Lemma \ref{L3} doest not hold.

\begin{Example}\label{E1} Let $I=(x_1x_3,x_1x_4,x_1x_6,x_2x_3,x_2x_4,x_3x_5,x_4x_5,x_4x_6,x_5x_6)$
be an ideal of $R=K[x_1,\ldots,x_6]$. Then $\reg(I)=2$ and so $I$ has a linear resolution and so has linear quotients. In particular, by Proposition \ref{L2}  it is weakly polymatroidal.
But $(I:x_2)=(x_3,x_4,x_1x_6,x_5x_6)$ does not have linear quotients, and so it is not weakly polymatroidal.
\end{Example}

\begin{Theorem}\label{P2}
Let $I$ be a monomial ideal of single degree $2$. Then the following conditions are equivalent:
\begin{enumerate}
\item[(i)] $I$ is (up to a relabeling of the variables of $R$ if necessarily) weakly polymatroidal;
\item[(ii)]  $I$ is vertex splittable;
\item[(iii)] $I$ has linear quotients;
\item[(iv)] $I$ has a linear resolution.
\end{enumerate}
\end{Theorem}

\begin{proof}
The implications $(i)\Longrightarrow(iii)\Longrightarrow(iv)$ and $(ii)\Longrightarrow (iii)$ are known by \cite[Theorem 1.4]{KH}, \cite[Theorem 3.2]{HHZ} and \cite[Theorem 2.4]{MK}.
Consider $(i)\Longrightarrow(ii)$ by Lemma \ref{L3}, $I_1=(I:x_1)$ is weakly polymatroidal of degree $1$ and so $I_1$ is
 vertex splittable. Now, we may consider $I=x_1I_1+I_2$ such that $I_2\subseteq I_1$ and by definition $I_2$ is weakly polymatroidal. Therefore by induction on the number of variables $I_2$ is vertex splittable.
 Thus $I$ is vertex splittable.

 It remains to prove $(iv)\Longrightarrow (i)$. We may assume $I=(x_1^2,\ldots,x_t^2,J)$ such that
$J$ is a square-free monomial ideal of single degree $2$. Polarizing the ideal $I$ yields the ideal $L=(x_1y_1,\ldots,x_ty_t,J)$ in $R=K[x_1,\ldots,x_n,y_1,\ldots,y_t]$. Since $I$ has a linear resolution, it implies that $L$ has a linear resolution and so it has linear quotients by \cite[Theorem 3.2]{HHZ}. Hence by Proposition \ref{L2}, $L$ is weakly polymatroidal and hence $x_ix_j\in J$ for all $1\leq i<j\leq t$. So $I$ is weakly polymatroidal.
\end{proof}

From the definition of weakly polymatroidal ideals, it is enough to exists an ordering such that the monomial ideal satisfy in the condition of definition. The following example is not weakly polymatroidal with ordering $x_1>x_2>x_3>x_4$ but by a new rebeling $x_3>x_1>x_2>x_4$, we obtain that $I$ is weakly polymatroidal. Hence the monomial ideal is not weakly polymatroidal when it does not satisfy the condition of definition by any ordering.

\begin{Example}
Let $I=(x_1 x_2, x_2 x_3, x_3 x_4)$ be an ideal of $R=K[x_1, x_2, x_3, x_4]$. By Macaulay2, $\reg(I)=2$ and so $I$ has a linear resolution. Hence by Proposition \ref{P2}, $I$ is weakly polymatroidal up to new relabeling of the variables of $R$ if necessarily. If we consider the ordering $x_1>x_2>x_3>x_4$ and the two elements $x_1 x_2$, $x_3 x_4$, then neither $x_1 (\frac{x_3 x_4}{x_3})$ nor $x_1 (\frac{x_3 x_4}{x_4})$ belongs to $I$. Hence by this ordering $I$ is not weakly polymatroidal. However by this argument we can not conclude that $I$ is not weakly polymatroidal.
\end{Example}

By the following example we conclude, in general, that Theorem \ref{P2} does not hold for monomial ideals of single degree bigger than $2$ and
also we show that the converse of Theorem \ref{T00} is not true, in general.
\begin{Example}\label{E2} Let $R=K[x_1,\ldots,x_6]$ and
\[I=(x_1x_2x_3,x_1x_2x_4,x_1x_2x_5,x_1x_2x_6,x_1x_4x_5,x_1x_5x_6,x_2x_3x_4,x_3x_4x_5,x_3x_4x_6,x_3x_5x_6,x_4x_5x_6)\]
be an ideal of $R$. Then
\[I^{\vee}=(x_1x_3x_4,x_1x_3x_5,x_1x_3x_6,x_1x_4x_5,x_1x_4x_6,x_2x_3x_5,x_2x_4x_5,x_2x_4x_6,x_2x_5x_6,x_3x_4x_5x_6).\]
Set $I^{\vee}=x_4I_1+I_2$, where\\
$I_1=(x_1x_3,x_1x_5,x_1x_6,x_2x_5,x_2x_6,x_3x_5x_6)$ and $I_2=(x_1x_3x_5,x_1x_3x_6,x_2x_3x_5,x_2x_5x_6)$.
 Therefore, it easily conclude that $I_2\subseteq I_1$ and $I_1,I_2$ are vertex splittable and so is $I^{\vee}$
 Now, we show that $I^{\vee}$ is not weakly polymatroidal. To do this,  we may consider the following cases:\\ (1) if $x_1>x_i$ for all $i\ne 1$ or $x_2>x_i$ for all $i\ne 2$, then we compare the two elements $u=x_1x_3x_4$ and $v=x_2x_5x_6$;\\ $(2)$ if $x_3>x_i$ for all $i\ne 3$, then we compare the two elements $u=x_2x_3x_5$ and $v=x_2x_4x_6$;\\ $(3)$ if $x_5>x_i$ for all $i\ne 5$, then we compare the two elements $u=x_2x_5x_6$ and $v=x_1x_3x_6$;
 $(4)$ if $x_6>x_i$ for all $i\ne 6$, then we compare the two elements $u=x_2x_5x_6$ and $v=x_1x_4x_5$;
 $(5)$ if $x_4>x_1>x_i$ for all $i\ne 1,4$, then we compare the two elements $u=x_1x_3x_6$ and $v=x_2x_5x_6$;\\
 $(6)$ if $x_4>x_2>x_i$ for all $i\ne 2,4$, then we compare the two elements $u=x_2x_4x_5$ and $v=x_1x_3x_4$;\\
 $(7)$ if $x_4>x_3>x_i$ for all $i\ne 3,4$, then we compare the two elements $u=x_1x_3x_4$ and $v=x_2x_4x_6$;\\
 $(8)$ if $x_4>x_5>x_i$ for all $i\ne 4,5$, then we compare the two elements $u=x_3x_4x_5x_6$ and $v=x_2x_3x_5$;\\
 $(9)$ if $x_4>x_6>x_i$ for all $i\ne 4,6$, then we compare the two elements $u=x_3x_4x_5x_6$ and $v=x_2x_3x_5$.\\
 Hence by using the above arguments, we immediately obtain that $I^{\vee}$ is not weakly polymatroidal.

\end{Example}

\subsection*{Acknowledgements}
The authors would like to express their deep gratitude to anonymous referees for their valuable comments which substantially improved the quality of the paper. The work of the second author has been supported financially by Vice-Chancellorship of Research and Technology, University of Kurdistan under research Project No. 99/11/19299.


\end{document}